\documentclass[reqno,draft]{amsproc}
\usepackage{amsthm}
\usepackage{amsfonts}
\usepackage{amsmath}
\usepackage{amssymb}
\usepackage{mathrsfs}
\usepackage{euscript}   
\usepackage{color}
\usepackage{mathbbol}
\usepackage{pifont}
\usepackage{tikz}

\makeatletter
\@namedef{subjclassname@2010}{%
\textup{2010} Mathematics Subject Classification}
\makeatother

\numberwithin{equation}{section}

\newtheorem{thm}{Theorem}[section]
\newtheorem{cor}[thm]{Corollary}
\newtheorem{lem}[thm]{Lemma}

\newtheorem*{opq*}{\bf Problem}
\newtheorem{dfn}[thm]{Definition}

\theoremstyle{remark}
\newtheorem{rem}[thm]{Remark}

\theoremstyle{definition}
\newtheorem{exa}[thm]{Example}

\newcommand*{\cbb}{\mathbb{C}}
\newcommand*{\D}{\mathrm{d\hspace{.1ex}}}
\newcommand*{\dbb}{\mathbb{D}}

\newcommand*{\E}{\mathrm{e}}
\newcommand*{\ecal}{\EuScript{E}}

\newcommand*{\gibh}{\mathscr{I}_{\hh_1,\hh_2}}
\newcommand*{\gqbh}{\mathscr{Q}_{\hh_1,\hh_2}}
\newcommand*{\gnbh}{\mathscr{N}_{\hh_1,\hh_2}}
\newcommand*{\gsbh}{\mathscr{S}_{\hh_1,\hh_2}}
\newcommand*{\gxbh}{\mathscr{X}_{\hh_1,\hh_2}}
\newcommand*{\Ge}{\geqslant}
\newcommand*{\hh}{\EuScript{H}}
\newcommand*{\I}{\mathrm{i}}

\newcommand*{\is}[2]{\langle#1,#2\rangle}
\newcommand*{\jd}[1]{\mathcal N(#1)}
\newcommand*{\kk}{\EuScript{K}}
\newcommand*{\Le}{\leqslant}
\newcommand*{\mcal}{\EuScript{M}}

\newcommand*{\ob}[1]{{\mathcal R}(#1)}
\newcommand*{\ogr}[1]{\boldsymbol B(#1)}

\newcommand*{\supp}{\mathrm{supp}\,}
\newcommand*{\tbb}{\mathbb{T}}

\begin{document}
   \title[Convergence of power sequences of B-operators with
applications] {Convergence of power sequences of
B-operators \\ with applications to stability
   }
   \author[S. Chavan]{Sameer Chavan}
   \address{Department of Mathematics and Statistics\\
Indian Institute of Technology Kanpur, India}
   \email{chavan@iitk.ac.in}
   \author[Z.\ J.\ Jab{\l}o\'nski]{Zenon Jan
Jab{\l}o\'nski}
   \address{Instytut Matematyki,
Uniwersytet Jagiello\'nski, ul.\ \L ojasiewicza 6,
PL-30348 Kra\-k\'ow, Poland}
\email{Zenon.Jablonski@im.uj.edu.pl}
   \author[I.\ B.\ Jung]{Il Bong Jung}
   \address{Department of Mathematics,
Kyungpook National University, Daegu 702-701, Korea}
   \email{ibjung@knu.ac.kr}
   \author[J.\ Stochel]{Jan Stochel}
\address{Instytut Matematyki, Uniwersytet
Jagiello\'nski, ul.\ \L ojasiewicza 6, PL-30348
Kra\-k\'ow, Poland} \email{Jan.Stochel@im.uj.edu.pl}
   \thanks{The research of the second and fourth
authors was supported by the National Science Center
(NCN) Grant OPUS No.\ DEC-2021/43/B/ST1/01651. The
research of the third author was supported by Basic
Science Research Program through the National Research
Foundation of Korea (NRF) funded by the Ministry of
Education (NRF-2021R111A1A01043569).}
   \begin{abstract}
The B-operators (abbreviation for Brownian-type
operators) are upper triangular $2\times 2$ block
matrix operators that satisfy certain algebraic
constraints. The purpose of this paper is to
characterize the weak, the strong and the uniform
stability of B-operators, respectively. This is
achieved by giving equivalent conditions for the
convergence of powers of a B-operator in each of the
corresponding topologies. A more subtle
characterization is obtained for B-operators with
subnormal $(2,2)$ entry. The issue of the strong
stability of the adjoint of a B-operator is also
discussed.
   \end{abstract}
   \maketitle

   \section{Introduction}
The stability of Hilbert (or Banach) space operators,
which can be thought of as a discrete version of the
stability of $C_0$-semigroups of operators, means that
the power sequence of a given operator converges to
the zero operator in one of the various possible
topologies. Among the topologies of interest to us are
the operator norm topology (uniform stability), the
strong operator topology (strong stability) and the
weak operator topology (weak stability). Stability in
its various aspects is related to model theory
\cite{Kub97}, the invariant subspace problem
\cite{Rad-Ros03}, supercyclicity \cite{Kub16},
ergodicity \cite{Eis10}, etc. Uniform and strong
stability is now well understood. However, this is not
the case for weak stability. Hence, very few results
on weak stability are known (see e.g.,
\cite[Sec.~II.3]{Eis10}).

   The B-operators introduced in \cite{C-J-J-S21}
under the name ``Brownian-type operators'' generalize
the Brownian unitaries \cite{Ag90}, the Brownian
isometries \cite{Ag-St95-6} and the quasi-Brownian
isometries \cite{Maj,C-J-J-S21}. The B-operators are
upper triangular $2\times 2$ block matrix operators
that satisfy certain algebraic constraints (see
Definition~\ref{defq}). It is worth mentioning that
upper triangular $2\times 2$ block matrices appear in
various parts of operator theory and functional
analysis (e.g., see
\cite{Fog64,Do-Pe72,Ag-St95-6,Pi97,WLee00,
Cur-Lee02,Mis17,JKP18,JKP19,C-J-J-S22}).

Our main goal in this paper is to give criteria for
the convergence of powers of a B-operator relative to
three topologies, namely the weak operator topology
({\sc wot}), the strong operator topology ({\sc sot})
and the operator norm topology (see
Theorems~\ref{puwsb} and \ref{niekj}). As a
consequence, we obtain the complete characterizations
of the stability of B-operators in terms of their
entries relative to any of these topologies (see
Corollaries~\ref{wstib} and \ref{niekj++}). It turns
out that there are no strongly stable B-operators.
However, there are many weakly stable B-operators. We
refer the reader to Section~\ref{Sec.5} for some
illustrative examples. We also pay special attention
to the case of B-operators with subnormal $(2,2)$
entry, called B-subnormal operators (see
Theorem~\ref{puwsa} and Corollary~\ref{strbli}). The
proofs of the main results are given in
Section~\ref{SuS2.2}, while the necessary preparatory
results are presented in Section~\ref{Sec.3}. In
Section~ \ref{Sec.4.5}, we discuss the question of
strong stability of the adjoint of a B-operator.

We begin by providing the necessary concepts and
facts. In what follows, $\cbb$ stands for the set of
complex numbers. We write $\tbb:=\{z\in \cbb \colon
|z|=1\}$ and $\mathbb{D}:=\{z\in \cbb \colon |z|<1\}$.
Given two (complex) Hilbert spaces $\hh$ and $\kk,$ we
denote by $\ogr{\hh,\kk}$ the Banach space of all
bounded linear operators from $\hh$ to $\kk$. The
kernel, the range and the modulus of an operator $T
\in \ogr{\hh,\kk}$ are denoted by $\jd{T}$, $\ob T$
and $|T|,$ respectively. We regard
$\ogr{\hh}:=\ogr{\hh,\hh}$ as a $C^*$-algebra. The
identity operator on $\hh$ is denoted by $I$. If $T\in
\ogr{\hh}$, then $r(T)$ stands for the spectral radius
of $T$. Recall that an operator $T\in \ogr{\hh}$ is
said to be {\em weakly}, {\em strongly} or {\em
uniformly stable} if the power sequence
$\{T^n\}_{n=1}^{\infty}$ converges to $0$ in the {\sc
wot}, the {\sc sot} or the operator norm topology,
respectively. We refer the reader to
\cite{Eis10,Kub97} for more information on this topic.
We say that $T\in \ogr{\hh}$ is {\em subnormal} if
there exist a Hilbert space $\kk$ and a normal
operator $N\in \ogr{\kk}$, called a {\em normal
extension} of $T$, such that $\hh\subseteq \kk$
(isometric embedding) and $Th = Nh$ for all $h \in
\hh$. A normal extension $N\in \ogr{\kk}$ of $T$ is
said to be {\em minimal} if $\kk$ is the only closed
subspace of $\kk$ containing $\hh$ and reducing $N$.
We call an operator $T\in \ogr{\hh}$ {\em hyponormal}
if $TT^* \Le T^*T$. We say that $T\in \ogr{\hh}$ is
{\em quasinormal} if $TT^*T=T^*TT$. By
\cite[Propositions~II.1.7 and II.4.2]{Co91}, we have
$\mathscr{U} \subseteq \mathscr{I} \subseteq
\mathscr{Q}$ and $\mathscr{U} \subseteq \mathscr{N}
\subseteq \mathscr{Q} \subseteq \mathscr{S} \subseteq
\mathscr{H}$, where $\mathscr{U}$, $\mathscr{I}$,
$\mathscr{N}$, $\mathscr{Q}$, $\mathscr{S}$ and
$\mathscr{H}$ denote the classes of unitary,
isometric, normal, quasinormal, subnormal and
hyponormal operators, respectively.

We are now ready to recall the concept of a
B-operator.
   \begin{dfn}[{\cite{C-J-J-S21}}] \label{defq}
   We say that $T\in \ogr{\hh}$ is a B-operator if it
has the block matrix form $T=\big[\begin{smallmatrix} V & E \\
0 & X \end{smallmatrix}\big]$ with respect to an
orthogonal decomposition $\hh=\hh_1 \oplus \hh_2$,
where the spaces $\hh_1$ and $\hh_2$ are nonzero and
the operators $V\in \ogr{\hh_1},$ $E\in
\ogr{\hh_2,\hh_1}$ and $X\in \ogr{\hh_2}$ satisfy the
following conditions{\em :}
   \allowdisplaybreaks
   \begin{gather} \notag
\text{$V$ is an isometry, i.e., $V^*V=I,$}
   \\  \label{gqb-2}
\text{$V^*E=0$, i.e., $\ob{V}\perp \ob{E}$,}
   \\ \label{gqb-3}
XE^*E=E^*EX.
   \end{gather}
If this is the case, we say that $T=\big[\begin{smallmatrix} V & E \\
0 & X \end{smallmatrix}\big]$ is a B-operator relative
to $\hh_1\oplus \hh_2$. If $E= 0$, then $T$ is called
block-diagonal. If $X$ is unitary $($resp., isometric,
normal, quasinormal, subnormal, hyponormal, etc.$)$,
then $T$ is called B-unitary $($resp., B-isometric,
B-normal, B-quasinormal, B-subnormal, B-hyponormal,
etc.$)$ relative to $\hh_1\oplus \hh_2$. If
$\mathscr{X}\subseteq \ogr{\hh_2}$ is a class of
operators, then the set of all B-operators of the form
$\big[\begin{smallmatrix} W & F \\
0 & X \end{smallmatrix}\big]$ relative to $\hh_1\oplus
\hh_2$, where $X\in \mathscr{X}$, is denoted by
$\mathscr{X}_{\hh_1,\hh_2}$.
   \end{dfn}
The reader should be aware that it can happen that a
B-operator has two different matrix forms. Namely,
there exists an operator which is B-quasinormal with
respect to two different orthogonal decompositions
(see \cite[Example~7.3]{C-J-J-S21}).

Observe that by the square root theorem (see
\cite[Theorem~2.4.4]{Sim-4}), the condition
\eqref{gqb-3} is equivalent to $X|E|=|E|X$, which
yields the following.
   \begin{align} \label{reotp}
   \begin{minipage}{67ex}
{\em If $T=\big[\begin{smallmatrix} V & E \\
0 & X \end{smallmatrix}\big]$ is a B-operator relative
to $\hh_1\oplus \hh_2$, then the spaces
$\hh_{21}:=\jd{E}$ and $\hh_{22}:=\overline{\ob{|E|}}$
reduce $X$, and $\hh_2=\hh_{21}\oplus \hh_{22}$.}
   \end{minipage}
   \end{align}

It is worth mentioning that the concept of the
B-operator is a far-reaching generalization of the
Brownian isometry \cite{Ag-St95-6} and the
quasi-Brownian isometry \cite{Maj,A-C-J-S19}, which
correspond to B-unitary and B-isometric operators,
respectively (see also \cite[Theorem~4.1]{A-C-J-S19}).
In other words, the Agler-Stankus concept of Brownian
isometry does not coincide with the notion of
B-isometry. On the other hand, the Agler-Stankus class
of Brownian unitaries forms a proper subclass of
B-unitaries.
   \section{Main results} Before stating the main results of
this paper, let us recall that if $V\in \ogr{\hh}$ is an
isometry, then the space
   \begin{align*}
\hh_{\mathrm{u}}:=\bigcap_{n=0}^{\infty} \ob{V^n}
   \end{align*}
reduces $V$ to a unitary operator, which is called the {\em
unitary part} of $V$ (see \cite[Theorem~I.1.1]{SF70}). We
begin with the case of general B-operators.
   \begin{thm} \label{puwsb}
Suppose that $T = \big[\begin{smallmatrix} V & E \\
0 & X \end{smallmatrix}\big]$ is a B-operator relative
to $\hh_1\oplus \hh_2$. Then the following conditions
are equivalent{\em :}
   \begin{enumerate}
   \item[(i)] $\{T^n\}_{n=1}^{\infty}$ is
{\sc wot}-convergent,
   \item[(ii)] $\{U^n\}_{n=1}^{\infty}$
and $\{X^n\}_{n=1}^{\infty}$ are {\sc wot}-convergent and
$\sup_{n\Ge 1} \|E_n\| < \infty$, where $U\in
\ogr{\hh_{\mathrm{1u}}}$ is the unitary part of $V$ and
$\{E_n\}_{n=1}^{\infty}\subseteq \ogr{\hh_2,\hh_1}$ is
given~by\footnote{The sequence $\{E_n\}_{n=1}^{\infty}$ is
related to the powers of $\big[\begin{smallmatrix} V & E \\
0 & X \end{smallmatrix}\big]$ (see Lemma~\ref{xyzzyx}).}
   \begin{align} \label{syru}
E_n = \sum_{j=0}^{n-1} V^jEX^{n-1-j}, \quad n\Ge 1.
   \end{align}
   \end{enumerate}
Moreover, if {\em (ii)} holds, $P:=\text{\sc
(wot)}\lim_{n\to\infty} U^n$ and $A:=\text{\sc
(wot)}\lim_{n\to\infty} X^n$, then $P$ is the
orthogonal projection of $\hh_{\mathrm{1u}}$ onto
$\jd{I-U}$ and
   \begin{align} \label{mytrfwb}
\textrm{$(${\sc wot}$)$} \lim_{n\to \infty} T^n =
\begin{bmatrix} P \oplus 0 & 0 \\ 0 & A
\end{bmatrix}.
   \end{align}
Further, if $\{U^n\}_{n=1}^{\infty}$ is {\sc
wot}-convergent and $r(X) < 1$, then {\em (ii)} holds
and $A=0$.
   \end{thm}
Using Theorem~\ref{puwsb}, we can characterize the weak
stability of B-operators.
   \begin{cor} \label{wstib}
Suppose that $T = \big[\begin{smallmatrix} V & E \\
0 & X \end{smallmatrix}\big]$ is a B-operator relative
to $\hh_1\oplus \hh_2$. Then the following conditions
are equivalent{\em :}
   \begin{enumerate}
   \item[(i)] $T$ is weakly stable,
   \item[(ii)] $U$ and
$X$ are weakly stable and $\sup_{n\Ge 1} \|E_n\| < \infty$,
where $U\in \ogr{\hh_{\mathrm{1u}}}$ is the unitary part of
$V$ and $\{E_n\}_{n=1}^{\infty}$ is as in \eqref{syru}.
   \end{enumerate}
   \end{cor}
The {\sc sot}-convergence of the power sequence
$\{T^n\}_{n=1}^{\infty}$ of a B-operator $T$ is described by
the following theorem.
   \begin{thm} \label{niekj}
Suppose that $T = \big[\begin{smallmatrix} V & E \\
0 & X \end{smallmatrix}\big]$ is a B-operator relative
to $\hh_1\oplus \hh_2$. Then the following conditions
are equivalent{\em :}
   \begin{enumerate}
   \item[(i)] $\{T^n\}_{n=1}^{\infty}$ is {\sc
sot}-convergent,
   \item[(ii)] $V=I$, $E=0$ and $\{X^n\}_{n=1}^{\infty}$ is {\sc
sot}-convergent.
   \end{enumerate}
Moreover, if {\em (ii)} holds, then
   \begin{align} \label{ctrex}
L:=\text{\sc (sot)} \lim_{n\to\infty} T^n =
\begin{bmatrix} I & 0 \\ 0 & A
\end{bmatrix},
   \end{align}
where $A:=\text{\sc (sot)}\lim_{n\to\infty} X^n$.
Furthermore, if $T\in \gsbh$, then $L\in \gsbh$.
   \end{thm}
   \begin{cor} \label{niekj++}
B-operators are neither strongly nor uniformly stable.
   \end{cor}
In the case of a B-subnormal operator, we obtain a
more subtle characterization of the {\sc
wot}-convergence of the power sequence
$\{T^n\}_{n=1}^{\infty}$.
   \begin{thm} \label{puwsa}
Let $T = \big[\begin{smallmatrix} V & E \\
0 & X \end{smallmatrix}\big]$ be a B-subnormal
operator relative to $\hh_1\oplus \hh_2$. Then the
following statements are equivalent{\em :}
   \begin{enumerate}
   \item[(i)] $\{T^n\}_{n=1}^{\infty}$ is
{\sc wot}-convergent,
   \item[(ii)] $\{U^n\}_{n=1}^{\infty}$
and $\{(X|_{\hh_{21}})^n\}_{n=1}^{\infty}$ are {\sc
wot}-convergent, $\sup_{n\Ge 1} \|E_n\| < \infty$ and
$\|X\|\Le 1$, where $U\in \ogr{\hh_{\mathrm{1u}}}$ is the
unitary part of $V$, $\{E_n\}_{n=1}^{\infty}$ is as in
\eqref{syru} and $\hh_{21}$ is as in \eqref{reotp}.
   \end{enumerate}
Moreover, if {\em (ii)} holds and $B:=\text{\sc
(wot)}\lim_{n\to\infty} (X|_{\hh_{21}})^n$, then
\eqref{mytrfwb} holds with $A=B \oplus 0$.
Furthermore, if $E$ is injective, then
$\hh_{21}=\{0\}$ and {\em (i)} is equivalent to
   \begin{enumerate}
   \item[(iii)] $\{U^n\}_{n=1}^{\infty}$
is {\sc wot}-convergent, $\sup_{n\Ge 1} \|E_n\| <
\infty$ and $\|X\|\Le 1$.
   \end{enumerate}
   \end{thm}
As a direct application of Theorem~\ref{puwsa}, we get the
following characterization of the weak stability of
B-subnormal operators (cf.\ Corollary~\ref{wstib}).
   \begin{cor}  \label{strbli}
Let $T = \big[\begin{smallmatrix} V & E \\
0 & X \end{smallmatrix}\big]$ be a B-subnormal
operator relative to $\hh_1\oplus \hh_2$. Then the
following conditions are equivalent{\em :}
   \begin{enumerate}
   \item[(i)] $T$ is weakly stable,
   \item[(ii)] $U$ and $X|_{\hh_{21}}$ are weakly stable, $\sup_{n\Ge 1}
\|E_n\| < \infty$ and $\|X\| \Le 1$, where $U\in
\ogr{\hh_{\mathrm{1u}}}$ is the unitary part of $V$,
$\{E_n\}_{n=1}^{\infty}$ is as in \eqref{syru} and $\hh_{21}$
is as in {\em \eqref{reotp}}.
   \end{enumerate}
Moreover, if $E$ is injective, then {\em (i)} is
equivalent to
   \begin{enumerate}
   \item[(iii)] $U$ is weakly stable, $\sup_{n\Ge 1}
\|E_n\| < \infty$ and $\|X\| \Le 1$.
   \end{enumerate}
   \end{cor}
   \section{\label{Sec.3}Preparatory lemmas}
We begin by recalling the concept of the semispectral
measure of a subnormal operator. Let $T\in \ogr{\hh}$
be a subnormal operator and $G$ be the spectral
measure of a mini\-mal normal extension $N\in
\ogr{\kk}$ of $T$. Set $F(\varDelta) =
PG(\varDelta)|_{\hh}$ for a Borel subset $\varDelta$
of $\cbb$, where $P\in \ogr{\kk}$ is the orthogonal
projection of $\kk$ onto $\hh$. Then $F$ is a Borel
semispectral measure on $\cbb$ (i.e., $F(\cbb)=I$ and
$F$ is $\sigma$-additive in the weak operator
topology) called the {\em semispectral measure} of
$T$. The definition of $F$ is independent of the
choice of a minimal normal extension of $T$ and $F$ is
a unique Borel semispectral measure on $\cbb$ such
that
   \begin{align} \label{mymor}
T^{*n}T^m=\int_{\cbb} z^m \bar z^n F(\D z), \quad m,n \Ge 0.
   \end{align}
For more information on this, the reader is referred
to \cite[Section~3]{Ju-St08} and \cite[Appendix]{St92}
(see also \cite{Bi-So87,Co91}).

We now describe the semispectral measures of the
components of the orthogonal sum of two subnormal
operators.
   \begin{lem}\label{sspor}
Let $T\in \ogr{\hh}$ be a subnormal operator, $F$ be the
semispectral measure of $T$ and $\mcal_1,\mcal_2$ be closed
subspaces of $\hh$ reducing $T$ such that $\hh=\mcal_1\oplus
\mcal_2$. Then for $j=1,2$, $\mcal_j$ reduces $F$ and
$F(\cdot)|_{\mcal_j}$ is the semispectral measure of
$T|_{\mcal_j}$.
   \end{lem}
   \begin{proof} Note that $T_j:=T|_{\mcal_j}$, $j=1,2$,
are subnormal operators and $T=T_1 \oplus T_2$.
Suppose that $N\in \ogr{\kk}$ is a minimal normal
extension of $T$ and $G$ is the spectral measure of
$N$. It follows from \cite[Lemma~3.4]{C-J-J-S22} that
$N=N_1\oplus N_2$, where $N_j\in \ogr{\kk_j}$ is the
minimal normal extension of $T_j$ for $j=1,2$. In
particular, $\kk=\kk_1 \oplus \kk_2$. For $j=1,2$, let
$F_j$ be the semispectral measure of $T_j$ and $G_j$
be the spectral measure of $N_j$. Then $G=G_1\oplus
G_2$. Hence, $F_1 \oplus F_2$ is a Borel semispectral
measure on $\cbb$ such that
   \allowdisplaybreaks
   \begin{align*}
\is{(F_1 \oplus F_2) (\varDelta)(f_1& \oplus f_2)}{f_1 \oplus
f_2} = \is{F_1 (\varDelta)f_1}{f_1} + \is{F_2
(\varDelta)f_2}{f_2}
   \\
& = \is{G_1 (\varDelta)f_1}{f_1} + \is{G_2
(\varDelta)f_2}{f_2}
   \\
& = \is{(G_1 \oplus G_2) (\varDelta)(f_1\oplus
f_2)}{f_1\oplus f_2}
   \\
& = \is{G (\varDelta)(f_1\oplus f_2)}{f_1\oplus f_2}
   \\
& = \is{F(\varDelta)(f_1\oplus f_2)}{f_1\oplus f_2}, \quad
f_j\in \mcal_j, \, j=1,2,
   \end{align*}
for any Borel subset $\varDelta$ of $\cbb$, so $F=F_1\oplus
F_2$. This completes the proof.
   \end{proof}
The next lemma is a direct consequence of the uniform
boundedness principle and \cite[Propositions~0.3 and
0.4]{Kub97}.
   \begin{lem} \label{wikwi}
If $T\in \ogr{\hh}$, then the following conditions are
equivalent{\em :}
   \begin{enumerate}
   \item[(i)] $r(T)<1$,
   \item[(ii)] the sequence $\{\sum_{j=0}^n
T^{*j}T^j\}_{n=0}^{\infty}$ is bounded in the {\sc wot},
   \item[(iii)] $\sum_{j=0}^{\infty}
\|T^{j}f\|^2 < \infty$ for all $f\in \hh$,
   \item[(iv)] the series  $\sum_{j=0}^{\infty}
T^{*j}T^j$ converges in the {\sc wot},
   \item[(v)] the
series $\sum_{j=0}^{\infty} T^{*j}T^j$ converges in the
operator norm.
   \end{enumerate}
   \end{lem}
The following description of the powers of a
B-operator can basically be proved as
\cite[Proposition~3.10]{C-J-J-S21}.
   \begin{lem} \label{xyzzyx}
Let $T = \big[\begin{smallmatrix} V & E \\
0 & X \end{smallmatrix}\big]$ be a B-operator relative
to $\hh_1\oplus \hh_2$ and $\{E_n\}_{n=1}^{\infty}$ be
as in \eqref{syru}. Then
   \begin{enumerate}
   \item[(i)] $T^n= \big[\begin{smallmatrix} V^n & E_n \\
0 & X^n \end{smallmatrix}\big]$ for $n\Ge 1$,
   \item[(ii)] $T^{*n}T^n= \Big[\begin{smallmatrix} I & 0 \\[.5ex]
0 & |E_n|^2 + |X^n|^2
\end{smallmatrix}\Big]$ for $n\Ge 1$,
   \item[(iii)] $E_n^*E_n =
E^*E \big(\sum_{j=0}^{n-1}X^{*j}X^j\big) =
\big(\sum_{j=0}^{n-1}X^{*j}X^j\big) E^*E$ for $n\Ge 1$.
   \end{enumerate}
   \end{lem}
The lemma below provides equivalent conditions for the
sequence $\{E_n\}_{n=1}^{\infty}$ to converge in the
{\sc wot} to $0$.
   \begin{lem} \label{cuscik}
Let $T = \big[\begin{smallmatrix} V & E \\
0 & X \end{smallmatrix}\big]$ be a B-operator relative
to $\hh_1\oplus \hh_2$ and $\{E_n\}_{n=1}^{\infty}$ be
as in \eqref{syru}. Then the following conditions are
equivalent{\em :}
   \begin{enumerate}
   \item[(i)]
$\text{\sc (wot)}\lim_{n\to\infty} E_n=0$,
   \item[(ii)] the sequence $\{E_n\}_{n=1}^{\infty}$ is {\sc
wot}-convergent,
   \item[(iii)] $\sup_{n\Ge
1} \|E_n\| < \infty$.
   \end{enumerate}
Moreover, if $r(X|_{\hh_{22}}) < 1$, where $\hh_{22}$
is as in {\em \eqref{reotp}}, then {\em (i)} is
satisfied.
   \end{lem}
   \begin{proof}
(i)$\Rightarrow$(ii) Obvious.

(ii)$\Rightarrow$(iii) Apply the uniform boundedness
principle.

(iii)$\Rightarrow$(i) In view of the Wold
decomposition theorem (see
\cite[Theorem~I.1.1]{SF70}), $\hh_1=\hh_{\mathrm{1u}}
\oplus \hh_{\mathrm{1s}}$, where $\hh_{\mathrm{1u}}$
reduces $V$ to a unitary operator, $\hh_{\mathrm{1s}}$
reduces $V$ to a completely non-unitary operator and
   \begin{align}
\label{setwq} \hh_{\mathrm{1s}} =
\bigoplus_{n=0}^{\infty} V^n (\jd{V^*}).
   \end{align}
Since, by \eqref{gqb-2}, $\ob{E} \subseteq \jd{V^*}$,
we deduce that
   \begin{align} \label{ibyt}
V^jEX^{n-1-j}f \in V^j (\jd{V^*}), \quad j\Ge 0, \, n
\Ge j+1, \, f \in \hh_2.
   \end{align}
Fix temporarily an integer $k\Ge 0$. Then
   \allowdisplaybreaks
   \begin{align} \notag
\is{E_n (f_1\oplus f_2)}{g} & \overset{\eqref{syru} }=
\Big\langle \sum_{j=0}^{n-1} V^jEX^{n-1-j} (f_1\oplus
f_2),g\Big\rangle
   \\ \notag
&\hspace{.6ex}\overset{(*)} = \is{V^k E X^{n-1-k}
(f_1\oplus f_2)}{g}
   \\ \notag
&\hspace{.8ex} = \is{X^{n-1-k} f_1}{E^* V^{*k}g} +
\is{V^k E X^{n-1-k} f_2}{g}
   \\ \notag
&\hspace{.6ex} \overset{(\dag)}= \is{X^{n-1-k}
f_1}{E^* V^{*k}g}, \quad f_1\in \ob{|E|}, \, f_2 \in
\jd{|E|},
      \\ \label{dwunast}
& \hspace{30ex} g \in V^k (\jd{V^*}),\, n \Ge k+1.
   \end{align}
where $(*)$ is a consequence of \eqref{setwq} and
\eqref{ibyt}, and $(\dag)$ follows from the fact that
$\jd{|E|}$ reduces $X$ and $\jd{E}=\jd{|E|}$. However,
by \eqref{gqb-3} and the square root theorem,
$X|E|=|E|X$, which together with
Lemma~\ref{xyzzyx}(iii) yields
   \allowdisplaybreaks
   \begin{align*}
\sum_{j=0}^{n} \|X^j |E|h\|^2 & = \Big \langle
\sum_{j=0}^{n} X^{*j}X^j E^*E h, h\Big\rangle =
\is{E_{n+1}^*E_{n+1}h}{h}
   \\
& \Le \big(\sup_{m\Ge 1} \|E_m\|\big)^2 \|h\|^2, \quad
h\in \hh_2, \, n\Ge 0.
   \end{align*}
This, together with (iii), implies that
   \begin{align} \label{insss}
\lim_{n\to \infty} \|X^n f\|=0, \quad f\in \ob{|E|}.
   \end{align}
Hence, by \eqref{dwunast}, we have
   \begin{align} \label{gfds}
\lim_{n\to\infty} \is{E_n (f_1\oplus f_2)}{g} = 0,
\quad f_1\in \ob{|E|}, \, f_2 \in \jd{|E|}, \, g\in
V^k (\jd{V^*}).
   \end{align}
It follows from \eqref{syru}, \eqref{setwq},
\eqref{ibyt} and the Wold decomposition theorem~that
   \begin{align} \label{ytred}
\is{E_n f}{g} = 0, \quad n \Ge 1, \, f\in \hh_2, \, g \in
\hh_{\mathrm{1u}}.
   \end{align}
Since $\ob{|E|}\oplus \jd{|E|}$ is dense in $\hh_2$,
combining \eqref{gfds} and \eqref{ytred} with the Wold
decomposition theorem and using (iii), we conclude that the
sequence $\{E_n\}_{n=1}^{\infty}$ {\sc wot}-converges to $0$.

It remains to show the ``moreover'' part. Note that by
\eqref{reotp}, the space $\hh_{22}$ reduces $X$.
Assume that $r(X|_{\hh_{22}})<1$. It is easy to see
that then the sequence $\{\|\sum_{j=0}^n
(X|_{\hh_{22}})^{*j}(X|_{\hh_{22}})^j\|\}_{n=0}^{\infty}$
is bounded. Hence by Lemma~\ref{xyzzyx}(iii),
$\sup_{n\Ge 1} \|E_n\| < \infty$. Using the
implication (iii)$\Rightarrow$(i) completes the~proof.
   \end{proof}
Regarding the ``moreover'' part of Lemma~\ref{cuscik},
let us observe the following.
   \begin{lem} \label{tywsr}
Let $T = \big[\begin{smallmatrix} V & E \\
0 & X \end{smallmatrix}\big]$ be a B-operator relative
to $\hh_1\oplus \hh_2$, $\{E_n\}_{n=1}^{\infty}$ be as
in \eqref{syru} and $\hh_{22}$ be as in {\em
\eqref{reotp}}. Suppose that $\ob{|E|}$ is closed.
Then the following conditions are equivalent{\em :}
   \begin{enumerate}
   \item[(i)] $\sup_{n\Ge 1} \|E_n\| <
\infty$,
   \item[(ii)] the series
$\sum_{j=0}^{\infty} (X|_{\hh_{22}})^{*j}(X|_{\hh_{22}})^j$
is {\sc wot}-convergent,
   \item[(iii)]
$r(X|_{\hh_{22}}) < 1$.
   \end{enumerate}
   \end{lem}
   \begin{proof}
(i)$\Rightarrow$(ii) Assume that $\sup_{n\Ge 1}
\|E_n\| < \infty$. By the inverse mapping theorem, the
operator $A:=|E||_{\hh_{22}}$ is invertible in
$\ogr{\hh_{22}}$. This, together with
Lemma~\ref{xyzzyx}(iii), implies~that
   \begin{align*}
\sum_{j=0}^n (X|_{\hh_{22}})^{*j}(X|_{\hh_{22}})^j =
(E_{n+1}^*E_{n+1})|_{\hh_{22}}A^{-2}, \quad n\Ge 0.
   \end{align*}
Hence $\sup_{n\Ge 0} \big\|\sum_{j=0}^n
(X|_{\hh_{22}})^{*j}(X|_{\hh_{22}})^j\big\| < \infty$.
Therefore, by Lemma~\ref{wikwi}, the series
$\sum_{j=0}^{\infty} (X|_{\hh_{22}})^{*j}(X|_{\hh_{22}})^j$
is {\sc wot}-convergent.

(ii)$\Leftrightarrow$(iii) This is a direct consequence of
Lemma~\ref{wikwi} applied to $X|_{\hh_{22}}$.

(iii)$\Rightarrow$(i) Apply the ``moreover'' part of
Lemma~\ref{cuscik}.
   \end{proof}
We will show in Example~\ref{nitclyst} that the
conclusion of Lemma~\ref{tywsr} is no longer valid if
the hypothesis that $\ob{|E|}$ is closed is deleted.

We also need the following relationship between the
weak stability of a unitary operator $U$ and the weak
convergence of the power sequence of $U$.
   \begin{lem} \label{ergyt}
Suppose that $U\in\ogr{\hh}$ is a unitary operator. Then the
following statements hold{\em :}
   \begin{enumerate}
   \item[(i)] if $\{U^n\}_{n=1}^{\infty}$ is {\sc
wot}-convergent, then $\text{\sc (wot)}\lim_{n\to\infty} U^n$
is the orthogonal projection of $\hh$ onto $\jd{I-U}$,
   \item[(ii)] $U$ is weakly stable if and only if
$\jd{I-U}=\{0\}$ and $\{U^n\}_{n=1}^{\infty}$ is {\sc
wot}-convergent.
   \end{enumerate}
   \end{lem}
   \begin{proof}
(i) Set $C=\text{\sc (wot)}\lim_{n\to \infty} U^n$. It
follows from the Ces\`{a}ro mean theorem that
   \begin{align} \label{wrrr}
\text{({\sc wot})} \lim_{n\to\infty} \frac{1}{n}
\sum_{j=0}^{n-1} U^j = C.
   \end{align}
In turn, by von Neumann's ergodic theorem (see
\cite[Theorem~II.11]{R-S-I80}),
   \begin{align*}
\text{({\sc sot})} \lim_{n\to\infty} \frac{1}{n}
\sum_{j=0}^{n-1} U^j = P,
   \end{align*}
where $P$ stands for the orthogonal projection of
$\hh$ onto $\jd{I-U}$. Combined with \eqref{wrrr},
this yields $C=P$.

(ii) This is an immediate consequence of (i).
   \end{proof}
   \begin{cor} \label{juzyr} Suppose that $U\in\ogr{\hh}$ is a unitary operator. Let
$U=\big[\begin{smallmatrix} J & 0\\0 &
U_0\end{smallmatrix}\big]$ be the matrix
representation of $U$ relative to the orthogonal
decomposition $\hh=\jd{I-U}\oplus \jd{I-U}^{\perp}$,
where $J$ is the identity operator on $\jd{I-U}$. Then
$\{U^n\}_{n=1}^{\infty}$ is {\sc wot}-convergent if
and only if $U_0$ is weakly stable.
   \end{cor}
   \section{\label{SuS2.2}Proofs of
the main results} We begin by proving
Theorem~\ref{puwsb}.
   \begin{proof}[Proof of Theorem~\ref{puwsb}]
(i)$\Leftrightarrow$(ii) Using Lemma~\ref{xyzzyx}(i), we
deduce that the sequence $\{T^n\}_{n=1}^{\infty}$ is {\sc
wot}-convergent if and only if the sequences
$\{V^n\}_{n=1}^{\infty}$, $\{E_n\}_{n=1}^{\infty}$ and
$\{X^n\}_{n=1}^{\infty}$ are {\sc wot}-convergent. In view of
Lemma~\ref{cuscik}, the sequence $\{E_n\}_{n=1}^{\infty}$ is
{\sc wot}-convergent if and only if $\sup_{n\Ge 1} \|E_n\| <
\infty$, and if this is the case then $\text{\sc
(wot)}\lim_{n\to \infty} E_n = 0$. It follows from the Wold
decomposition theorem that $W$ is weakly stable, where
$W:=V|_{\hh_1\ominus \hh_{1\mathrm{u}}}$. Hence, because
$V=U\oplus W$, the sequence $\{V^n\}_{n=1}^{\infty}$ is {\sc
wot}-convergent if and only if the sequence
$\{U^n\}_{n=1}^{\infty}$ is {\sc wot}-convergent. This
implies that the conditions (i) and (ii) are equivalent.

Suppose now that (ii) is satisfied. It follows from
Lemma~\ref{ergyt}(i) that $P:=\text{\sc (wot)}\lim_{n\to
\infty} U^n$ is the orthogonal projection of $\hh$ onto
$\jd{I-U}$. From what was shown in the previous paragraph, we
can deduce that \eqref{mytrfwb} is valid.

To prove the ``furthermore'' part, assume that
$\{U^n\}_{n=1}^{\infty}$ is {\sc wot}-convergent and
$r(X) < 1$. Then $\lim_{n \to \infty} X^n=0$ and
$r(X|_{\hh_{22}})<1$, so by Lemma~\ref{cuscik} the
condition (ii) is satisfied and consequently
\eqref{mytrfwb} holds with $A=0$.
   \end{proof}
   Theorem~\ref{niekj} is basically a consequence of
the proof of Theorem~\ref{puwsb}.
   \begin{proof}[Proof of Theorem~\ref{niekj}]
(i)$\Rightarrow$(ii) By assumption and Lemma~\ref{xyzzyx}(i),
the pow\-er sequences $\{V^n\}_{n=1}^{\infty}$ and
$\{X^n\}_{n=1}^{\infty}$ are {\sc sot}-convergent. Let
$V=U\oplus W$ be the Wold decomposition of $V$, where $U\in
\ogr{\hh_{1\mathrm{u}}}$ and $W\in
\ogr{\hh_{1\mathrm{u}}^{\perp}}$ are the unitary and the
completely non-unitary parts of $V$, respectively. Clearly,
the power sequences $\{U^n\}_{n=1}^{\infty}$ and
$\{W^n\}_{n=1}^{\infty}$ are {\sc sot}-convergent. By the
Wold decomposition theorem, $W$ is weakly stable, and
consequently $W$ is strongly stable. Since
   \begin{align} \label{drwa}
   \begin{minipage}{50ex}
{\em the {\sc sot}-limit of isometries is an
isometry,}
   \end{minipage}
   \end{align}
we conclude that $\hh_{1\mathrm{u}}=\hh_{1}$ and $V=U$ is
unitary. It follows from Lemma~\ref{ergyt}(i) that $\text{\sc
(sot)}\lim_{n\to\infty} V^n$ is the orthogonal projection of
$\hh_1$ onto $\jd{I-V}$. Combined with \eqref{drwa}, this
yields $V=I$. Hence, by \eqref{gqb-2}, $E= 0$, which shows
that (ii) holds.

(ii)$\Rightarrow$(i) Since $T=
\big[\begin{smallmatrix} I & 0 \\ 0 & X
\end{smallmatrix}\big]$, we deduce that
\eqref{ctrex} holds. An application of Bishop's
theorem (see \cite[Theorem~II.1.17]{Co91}) justifies
the ``furthermore'' part.
   \end{proof}
   \begin{rem}
An inspection of the proof of Theorem~\ref{niekj}
shows that if $V$ is an isometry, then
$\{V^n\}_{n=1}^{\infty}$ is {\sc sot}-convergent if
and only if $V=I$.
   \hfill $\diamondsuit$
   \end{rem}
Before proving Theorem~\ref{puwsa}, we need an auxiliary
lemma. First, observe that an application of the uniform
boundedness principle, together with Gelfand's formula for
the spectral radius and the fact that the norm and the
spectral radius of a subnormal operator coincide (see
\cite[Propositions~II.4.2 and II.4.6]{Co91}), leads to the
following assertion.
   \begin{align} \label{wyudr}
   \begin{minipage}{73ex}
{\em If $T\in \ogr{\hh}$ is such that the power
sequence $\{T^n\}_{n=1}^{\infty}$ is {\sc
wot}-convergent, then $r(T)\Le 1${\em ;} if in
addition $T$ is subnormal, then $\|T\|\Le 1$}.
   \end{minipage}
   \end{align}
This means that if $T = \big[\begin{smallmatrix} V & E \\
0 & X \end{smallmatrix}\big]$ is B-subnormal relative
to $\hh_1\oplus \hh_2$, then the contractivity of $X$
is necessary for each of the power sequences
$\{X^n\}_{n=1}^{\infty}$ and $\{T^n\}_{n=1}^{\infty}$
to converge in the {\sc wot}.
   \begin{lem} \label{nubt}
Let $T = \big[\begin{smallmatrix} V & E \\
0 & X \end{smallmatrix}\big]$ be a B-subnormal
operator relative to $\hh_1\oplus \hh_2$, $F$ be the
semispectral measure of $X$ and $\hh_{22}$ be as in
{\em \eqref{reotp}}. Suppose that $\|X\|\Le 1$ and
$\sup_{n\Ge 1} \|E_n\| < \infty$. Then
   \begin{enumerate}
   \item[(i)] $\hh_{22}$ reduces both $X$ and
$F$, $F(\cdot)|_{\hh_{22}}$ is the semispectral
measure of $X|_{\hh_{22}}$ and
$F(\tbb)|_{\hh_{22}}=0$,
   \item[(ii)] $X|_{\hh_{22}}$ is weakly stable,
   \item[(iii)] if $E$ is
injective, then $X$ is weakly stable.
   \end{enumerate}
   \end{lem}
   \begin{proof}
(i) By \eqref{reotp}, the space $\hh_{22}$ reduces
$X$. In view of Lemma~\ref{sspor} applied to
$\mcal_1=\hh_{21}$ and $\mcal_2=\hh_{22}$, it remains
to show that $F(\tbb)|_{\hh_{22}}=0$. It follows from
\eqref{insss} that $\lim_{n\to \infty} \|X^n f\|=0$
for $f\in \ob{|E|}$. In turn, by
\cite[Proposition~4.2(i)]{J-J-S22}, we have
$\lim_{n\to \infty} \|X^n f\|^2= \is{F(\tbb)f}{f}$ for
$f\in \hh_2$. Combining these two equalities, we see
that $\is{F(\tbb)f}{f}=0$ for all $f\in \ob{|E|}$,
which yields $F(\tbb)|_{\hh_{22}}=0$.

(ii) Let $F_2$ be the semispectral measure of
$X_2:=X|_{\hh_{22}}$, $N_2$ be a minimal normal extension of
$X_2$ and $G_2$ be the spectral measure of $N_2$. Since by
\cite[Corollary~II.2.17]{Co91}, $\|N_2\|=\|X_2\|\Le 1$, we
see that $G_2$ and consequently $F_2$ are supported in
$\overline{\mathbb{D}}$. It follows from (i) that
$F_2(\tbb)=0$. Hence, applying \eqref{mymor} to $T=X_2$ and
$F=F_2$, we see that
   \begin{align*}
\is{X_2^n f}{f} = \int_{\mathbb{D}} z^n \is{F_2(dz)f}{f},
\quad n \Ge 0, \, f\in \hh_{22}.
   \end{align*}
Using the Lebesgue dominated convergence theorem, we conclude
that (ii) holds.

(iii) If $E$ is injective and consequently
$\jd{|E|}=\{0\}$, we see that $\hh_{22} = \hh_2$, and
so (ii) implies (iii).
   \end{proof}
   \begin{proof}[Proof of Theorem~\ref{puwsa}]
That the conditions (i) and (ii) are equivalent follows from
\eqref{wyudr}, Lemma~\ref{nubt}(ii) and Theorem~\ref{puwsb}.
If (ii) is satisfied, then by Lemma~\ref{nubt}(ii),
$\text{\sc (wot)} \lim_{n\to\infty} X^n=B\oplus 0$, and thus,
in view of Theorem~\ref{puwsb}, the equality \eqref{mytrfwb}
holds with $A=B \oplus 0$. The ``furthermore'' part is
obvious.
   \end{proof}
   \section{\label{Sec.4.5}Strong stability of $T^*$}
We begin with a general criterion for strong stability
of the adjoint of an operator. Note that due to the
uniform boundedness principle the power boundedness of
an operator $T\in \ogr{\hh}$ is a necessary condition
for $T^*$ to be strongly stable. Recall that an
operator $T\in \ogr{\hh}$ is said to be {\em analytic}
if $\bigcap_{n=0}^{\infty} \ob{T^n}=\{0\}$.
   \begin{lem} \label{srtrsde}
If $T\in \ogr{\hh}$ is a power bounded operator such
that $\bigcap_{n=0}^{\infty}
\overline{\ob{T^n}}=\{0\}$, then $T^*$ is strongly
stable. In particular, this is the case when $T$ is
power bounded, left-invertible and analytic.
   \end{lem}
   \begin{proof}
Suppose that $T\in \ogr{\hh}$. First, observe that
   \begin{align} \label{vidg}
\bigcap_{n=0}^{\infty} \overline{\ob{T^n}} = \Big(
\bigvee_{n=0}^{\infty} \jd{T^{*n}}\Big)^{\perp}.
   \end{align}
It easy to see that $\lim_{n\to \infty} \|T^{*n}h\|=
0$ for every $h \in \mathcal E$, where $\mathcal E =
\bigcup_{n=0}^{\infty} \jd{T^{*n}}$ (clearly,
$\mathcal E$ is a vector space). If $T$ is power
bounded, then $\lim_{n\to \infty} \|T^{*n}h\|= 0$ for
every $h\in \overline {\mathcal E}$. This combined
with \eqref{vidg} and the fact that powers of
left-invertible operators being left-invertible have
closed range completes the proof.
   \end{proof}
Before describing $\overline{\ob{T^n}}$ for a
B-operator $T$, we need the following fact.
   \begin{align} \label{pywurs}
   \begin{minipage}{72ex}
{\em Fix a positive integer $n$. If $T\in \ogr{\hh}$,
then $T$ is left-invertible if and only if $T^n$ is
left-invertible. Consequently, if $T$ is injective,
then $T$ has closed range if and only if $T^n$ has
closed range.}
   \end{minipage}
   \end{align}
Indeed, the first statement can be proved purely
algebraically. The second statement follows from the
first and the well-known fact that $T$ is
left-invertible if and only if $T$ is injective and
has closed range.
   \begin{lem} \label{tttp}
Let $T = \big[\begin{smallmatrix} V & E \\
0 & X \end{smallmatrix}\big]$ be a B-operator relative
to $\hh_1\oplus \hh_2$ and \mbox{$n\Ge 1$}. Set
$S_n=\Big[\begin{smallmatrix} I & 0 \\[.3ex] 0
& \sqrt{|E_n|^2+|X^n|^2}\end{smallmatrix}\Big]$. Then
there exists a unique unitary isomorphism $U\colon
\overline{\ob{T^n}} \to \overline{\ob{S_n}}$ such that
$UT^n=S_n$. Moreover, the following are
\mbox{equivalent}{\em :}
   \begin{enumerate}
   \item[(i)] $\ob{T^n}$ is closed,
   \item[(ii)] $\ob{\sqrt{|E_n|^2+|X^n|^2}\/}$ is closed,
   \item[(iii)] $\ob{|E_n|}+\ob{|X^n|}$ is closed,
   \item[(iv)] $\ob{E_n^*}+\ob{X^{*n}}$ is closed.
   \end{enumerate}
Furthermore, if $E$ or $X$ is injective, then
$\ob{T^n}$ is closed if and only if $\ob{T}$ is
closed, or equivalently if and only if $|E|^2+|X|^2$
is invertible in $\ogr{\hh_2}$.
   \end{lem}
   \begin{proof}
It follows from \eqref{gqb-2} and \eqref{syru} that
$V^{*n}E_n=0$, or equivalently that $\ob{V^n} \perp
\ob{E_n}$. This, together with Lemma~\ref{xyzzyx}(i)
yields
   \allowdisplaybreaks
   \begin{align*}
\|T^n (h_1\oplus h_2)\|^2 & = \|V^nh_1 \oplus E_nh_2
\oplus X^nh_2\|^2
   \\
&= \|h_1\|^2 + \|E_nh_2\|^2 + \|X^nh_2\|^2
   \\
& = \|S_n (h_1\oplus h_2)\|^2, \quad h_1\in \hh_1,\,
h_2 \in \hh_2.
   \end{align*}
Therefore, there exists a unique unitary isomorphism
$U\colon \overline{\ob{T^n}} \to \overline{\ob{S_n}}$
such that $UT^n=S_n$. As a consequence, the conditions
(i) and (ii) are equivalent. Since by the theorem of
T. Crimmins (see \cite[Theorem~2.2]{Fil-Wil71})
$\ob{A} + \ob{B}=\ob{\sqrt{AA^*+BB^*}}$ for $A,B\in
\ogr{\hh}$, we deduce that the conditions (ii)-(iv)
are equivalent.

Now we prove the ``furthermore'' part. Suppose that
$E$ or $X$ is injective. Then $D:=|E|^2+|X|^2$ is
injective. Hence, by Lemma~\ref{xyzzyx}(ii), $T$ is
injective. Therefore, by \eqref{pywurs}, $T$ has
closed range if and only if $T^n$ has closed range.
Finally, $T$ has closed range if and only if $T$ is
left-invertible, or equivalently if and only if $T^*T$
is invertible. By Lemma~\ref{xyzzyx}(ii), this is
equivalent to the invertibility of $D$.
   \end{proof}
It may happen that for a B-operator $T$, $T^2$ has
closed range, but $T$ does not.
   \begin{exa}
Let $\hh$ be an infinite-dimensional Hilbert space.
Set $\hh_1=\hh_2=\hh\oplus \hh$. Let $A,C\in
\ogr{\hh}$ be operators such that $\ob{A^*}$ is
closed, $\ob{C^*}$ is not closed and $A^*AC=0$ (e.g.,
consider $A=A_1\oplus 0$ and $C=0\oplus C_1$ relative
to $\hh=\mcal \oplus \mcal$ such that $\ob{A_1^*}$ is
closed and $\ob{C_1^*}$ is not closed). Take an
isometry $V\in\ogr{\hh_1}$ such that $\ob{V}
\subseteq \{0\} \oplus \hh$. Set $E=\big[\begin{smallmatrix} A & 0  \\
0 & 0\end{smallmatrix}\big]$ and  $X=\big[\begin{smallmatrix} 0 & C  \\
0 & 0\end{smallmatrix}\big]$ (relative to $\hh\oplus
\hh$),
and $T=\big[\begin{smallmatrix} V  & E  \\
0 & X\end{smallmatrix}\big]$ (relative to $\hh_1\oplus
\hh_2$). It is easy to see that $T$ is a B-operator.
Since $\ob{Z^*}=\ob{|Z|}$ for any operator $Z$ and
$\sqrt{|E|^2 + |X|^2} =
\Big[\begin{smallmatrix} |A| & 0  \\
0 & |C|\end{smallmatrix}\Big]$, we infer from
Lemma~\ref{tttp} that $\ob{T}$ is not closed. In turn,
by Lemma~\ref{xyzzyx}(iii) and the commutativity of
$|X|$ and $|E|$, we have
   \begin{align*}
\sqrt{|E_2|^2 + |X^2|^2} = |E_2| = (I +
|X|^2)^{1/2}|E| =
\Big[\begin{smallmatrix} |A| & 0  \\
0 & 0\end{smallmatrix}\Big].
   \end{align*}
Hence, by Lemma~\ref{tttp} again, $\ob{T^2}$ is
closed.
   \end{exa}
   \begin{rem}
Regarding Lemma~\ref{tttp}, let us observe that $X$ is
injective if and only if $X^n$ is injective for some
(equivalently, for every) positive integer $n$. In
turn, by Lemma~\ref{xyzzyx}(iii), $E$ is injective if
and only if $E_n$ is injective for some (equivalently,
for every) positive integer $n$. Moreover,
$|E|^2+|X|^2$ is invertible if and only if
$|E_n|^2+|X^n|^2$ is invertible for some
(equivalently, for every) positive integer $n$.
Indeed, by Lemma~\ref{xyzzyx}(ii), $|E|^2+|X|^2$ is
invertible if and only if $T^*T$ is invertible, or
equivalently if and only if $T$ is left-invertible. In
view of \eqref{pywurs}, the latter is equivalent to
the left-invertibility of $T^n$ and consequently the
invertibility of $T^{*n}T^n$. Applying
Lemma~\ref{xyzzyx}(ii) completes the proof.
   \hfill $\diamondsuit$
   \end{rem}
We now provide sufficient conditions for the adjoint
of a B-operator to be strongly stable (see
Remark~\ref{nycuxa} for necessary conditions).
   \begin{thm} \label{crytstyr}
Let $T = \big[\begin{smallmatrix} V & E \\
0 & X \end{smallmatrix}\big]$ be a B-operator relative
to $\hh_1\oplus \hh_2$ such that $V$ is completely
non-unitary, $\sup_{n\Ge 1} \|E_n\| < \infty$ and $X$
is power bounded. Suppose that $X$ is injective and
analytic, and $|E|^2+|X|^2$ is invertible in
$\ogr{\hh_2}$. Then $T^*$ is strongly stable.
   \end{thm}
   \begin{proof}
In view of Lemma~\ref{srtrsde}, to complete the proof,
it suffices to show that $T$ is power bounded,
left-invertible, and analytic. It follows from
Lemma~\ref{xyzzyx}(i) that $T$ is power bounded if and
only if $\sup_{n\Ge 1} \|E_n\| < \infty$ and $X$ is
power bounded. In turn, $T$ is left-invertible if and
only if $T^*T$ is invertible, or equivalently by
Lemma~\ref{xyzzyx}(ii), the sum $|E|^2+|X|^2$ is
invertible. It remains to show that $T$ is analytic.
Take $h_1\oplus h_2 \in \bigcap_{n=0}^{\infty}
\ob{T^n}$, where $h_1\in \hh_1$ and $h_2\in \hh_2$.
Then there exist sequences
$\{h_{1,n}\}_{n=1}^{\infty}\subseteq \hh_1$ and
$\{h_{2,n}\}_{n=1}^{\infty}\subseteq \hh_2$ such that
$h_1\oplus h_2 = T^n (h_{1,n}\oplus h_{2,n})$ for
every integer $n\Ge 1$. According to
Lemma~\ref{xyzzyx}(i), we have
   \begin{align} \label{hyntew}
h_1=V^n h_{1,n} + E_n h_{2,n} \quad \text{and} \quad
h_2=X^n h_{2,n} \quad \text{for every integer $n\Ge
1$.}
   \end{align}
By the analyticity of $X$, we conclude that $h_2=0$.
Using \eqref{gqb-3} and \eqref{hyntew}, we get
   \begin{align*}
X^nE^*E h_{2,n} = E^*E X^n h_{2,n} = E^*E h_{2} =0,
\quad n\Ge 1.
   \end{align*}
Knowing that $X^n$ is injective, we see that $E
h_{2,n}=0$. Hence, by Lemma~\ref{xyzzyx}(iii), $E_n
h_{2,n}=0$. This, together with \eqref{hyntew},
implies that $h_1=V^n h_{1,n}$ for all $n \Ge 1$. It
follows from the Wold decomposition theorem that $V$
being completely non-unitary is analytic, so $h_1=0$.
Consequently, $T$ is analytic.
   \end{proof}
   \begin{rem} \label{nycuxa}
Regarding Theorem~\ref{crytstyr}, we note that the
assumptions that $V$ is completely non-unitary,
$\sup_{n\Ge 1} \|E_n\| < \infty$ and $X$ is power
bounded are necessary for the adjoint of a B-operator
$T = \big[\begin{smallmatrix} V & E \\
0 & X \end{smallmatrix}\big]$ to be strongly stable.
Indeed, assume that $T^*$ is strongly stable. Then, by
Lemma~\ref{xyzzyx}(i), $T^{*n}=
\big[\begin{smallmatrix} V^{*n} & 0 \\
E_n^* & X^{*n} \end{smallmatrix}\big]$ for $n\Ge 1$.
Thus, by the uniform boundedness principle, $X$ is
power bounded, $\sup_{n\Ge 1} \|E_n\| < \infty$, and
$\text{\sc (sot)}\lim_{n\to\infty} V^{*n}=0$. Using
the Wold decomposition theorem, we conclude that $V$
is completely non-unitary.

To have an example illustrating
Theorem~\ref{crytstyr}, consider three isometries
$V\in \ogr{\hh_1}$, $W_1\in \ogr{\hh_2,\hh_1}$ and
$W_2\in \ogr{\hh_2}$ such that $V$ and $W_2$ are
completely non-unitary and $\ob{V}\perp \ob{W_1}$. Let
$z,w$ be two nonzero complex numbers such that
$|w| < 1$. Set $E=z W_1$ and $X=w W_2$. Then $T = \big[\begin{smallmatrix} V & E \\
0 & X\end{smallmatrix}\big]$ is a B-operator relative
to $\hh_1\oplus \hh_2$, which is not strongly stable
(see Corollary~\ref{niekj++}), which is not
block-diagonal and whose adjoint $T^*$ is strongly
stable (use the Wold decomposition theorem and
Theorem~\ref{crytstyr}).
   \hfill $\diamondsuit$
   \end{rem}
   \section{\label{Sec.5}Examples} Regarding Theorems~\ref{puwsb} and
\ref{puwsa}, it is worth mentioning that in general
the power sequence $\{U^n\}_{n=1}^{\infty}$ is not
{\sc wot}-convergent. The trivial reason is that if
the sequence $\{U^n\}_{n=1}^{\infty}$ is {\sc
wot}-convergent to a nonzero operator (e.g., $U=I$),
then $\{(-U)^n\}_{n=1}^{\infty}$ is not {\sc
wot}-convergent. Non-trivial examples related to the
{\sc wot}-convergence of the power sequence
$\{U^n\}_{n=1}^{\infty}$ are given below. We refer the
reader to Corollary~\ref{juzyr} for the relationship
between the {\sc wot}-convergence of
$\{U^n\}_{n=1}^{\infty}$ and the weak stability of $U$
for a unitary $U$.
   \begin{exa}
Fix an integer $m\Ge 2$ and consider a unitary operator $U\in
\ogr{\hh}$ whose spectrum equals $\{z_0, \ldots, z_{m-1}\}$,
the set of all $m$th roots of $1$. Since the latter set is a
cyclic multiplicative subgroup of $\tbb$, we can assume that
$z_j=z_1^j$ for $j=0, \dots, m-1$. Then, by the spectral
theorem $U=\sum_{j=0}^{m-1} z_1^j P_j$, where
$\{P_j\}_{j=0}^{m-1}$ are nonzero orthogonal projections such
that $\hh=\bigoplus_{j=0}^{m-1} \ob{P_j}$. Thus, we have
   \begin{align*}
\is{U^{km+j}f}{f} = z_1^{j} \|f\|^2, \quad k \Ge 0, \,
j\in \{0, \ldots, m-1\}, \, f\in \ob{P_1}.
   \end{align*}
Hence, if $f$ is a nonzero vector in $\ob{P_1}$, then
the sequence $\{\is{U^n f}{f}\}_{n=1}^{\infty}$ has
$m$ distinct cluster points, showing that
$\{U^n\}_{n=1}^{\infty}$ is not {\sc wot}-convergent.

It could be even worse, namely, if $U=zI$, where
$z=\E^{2\pi\I \vartheta}$ and $\vartheta$ is an irrational
number, then by Jacobi's theorem (see
\cite[Theorem~I.3.13]{Dev03}) the closure of the set
$\{\is{U^n f}{f}\}_{n=1}^{\infty}$ is equal to $\tbb$ for
every normalized vector $f\in \hh$.

In turn, it can happen that the sequence
$\{U^n\}_{n=1}^{\infty}$ does converge in the {\sc wot},
e.g., when the spectral measure $G$ of a unitary operator
$U\in \ogr{\hh}$ is absolutely continuous with respect to
$\lambda$, the normalized Lebesgue measure on $\tbb$. Indeed,
by the spectral theorem and the Radon-Nikodym theorem,
   \begin{align*}
\is{U^n f}{f} = \Big\langle\int_{\tbb} z^n G(\D
z)f,f\Big\rangle =\int_{\tbb} z^n \phi_f(z) \lambda(\D z),
\quad n \Ge 0, \, f\in \hh,
   \end{align*}
where $\phi_f\in L^1(\lambda)$ is the Radon-Nikodym
derivative of $\is{G(\cdot)f}{f}$ with respect
to~$\lambda$. Applying the Riemann-Lebesgue lemma (see
\cite[Theorem~5.15]{Rud}) shows that $U$ is weakly
stable. A simple example of a unitary operator with
these properties is the operator of multiplication by
the independent variable ``$z$'' on
$L^2(\tbb,\lambda)$. This kind of result is related to
the characterization of weak stability via Rajchman
measures (see \cite[Proposition~IV.1.5]{Eis10}). We
also refer the reader to \cite{Kub16} for more
information on the weak stability of unitary operators
and its relationship to supercyclicity. In particular,
it is shown there that there are singular-continuous
unitary operators that are weakly stable, as well as
singular-continuous unitary operators that are weakly
unstable.
   \hfill $\diamondsuit$
   \end{exa}
Referring to the previous example, it is worth making
the following remark.
   \begin{rem}
As we know, every completely non-unitary isometry is
weakly stable. In fact, the same property is shared by
completely non-unitary contractions, namely, each such
operator is weakly stable (see
\cite[Corollary~7.4]{Kub97}). On this occasion, let us
recall that every completely non-unitary cohyponormal
contraction is strongly stable (see
\cite[Theorem~3]{Put75}; see also
\cite[Theorem]{Kub-Vie94}). Another result in this
spirit can be found in
\cite[Corollary~3.6(a)]{Chav08}.
   \hfill $\diamondsuit$
   \end{rem}
We will now discuss the issue of weak stability of
B-isometries (which are called quasi-Brownian
isometries in \cite{A-C-J-S19} and
$\triangle_T$-regular $2$-isometries in \cite{Maj}).
   \begin{exa} \label{bsaq}
Take two infinite dimensional Hilbert spaces $\hh_1$ and
$\hh_2$ such that $\dim \hh_1 \Ge \dim \hh_2$. Let $V\in
\ogr{\hh_1}$ and $X\in \ogr{\hh_2}$ be isometries and
$E\in\ogr{\hh_2,\hh_1}$ be such that $\ob{V}\perp \ob{E}$
and $XE^*E=E^*EX$. Then $T := \big[\begin{smallmatrix} V & E \\
0 & X \end{smallmatrix}\big] \in \gibh$. It follows
from Lemma~\ref{xyzzyx}(iii) that $E_n^*E_n = n E^*E$
for every $n\Ge 1$. This implies that $\sup_{n\Ge 1}
\|E_n\| < \infty$ if and only $E=0$. Hence, by
Corollary~\ref{wstib}, $T$ is weakly stable if and
only if $E=0$ and the unitary parts of $V$ and $X$ are
weakly stable (recall that $\|T^*T-I\|^{1/2}=\|E\|$ is
called the covariance of $T$). In particular, if $T$
is Brownian isometry of covariance $\sigma > 0$ (in
the Agler-Stankus terminology, see
\cite[Proposition~5.37]{Ag-St95-6}), that is, $V$ is
isometric, $X$ is unitary and $E=\sigma W$, where
$W\in \ogr{\hh_2,\hh_1}$ is an injective contraction
such that $\ob{V} \perp \ob{W}$ and $XW^*W=W^*WX$,
then $T$ is never weakly stable.
   \hfill $\diamondsuit$
   \end{exa}
Next, we will give an example of a weakly stable
B-normal operator which is not block-diagonal. In this
example, $E$ can be chosen so that $E$ is injective
and $\ob{|E|}$ is closed.
   \begin{exa}[Example~\ref{bsaq} continued]  \label{hgaq}
Let $V, E, X$ be as in Example~\ref{bsaq}. Assume
furthermore that $V$ is completely non-unitary, $E$ is
injective and $X$ is unitary. Let $z$ be a complex
number
such that $0<|z|<1$. Then clearly $T_z := \big[\begin{smallmatrix} V & E \\
0 & zX\end{smallmatrix}\big] \in \gnbh$. By
\cite[Proposition~3.10]{C-J-J-S21},
$\{T_z^n\}_{n=1}^{\infty} \subseteq \gqbh$. As powers
of normal operators are normal, we deduce from
Lemma~\ref{xyzzyx}(i) that $\{T_z^n\}_{n=1}^{\infty}
\subseteq \gnbh$. Knowing that $\|zX\| < 1$, the
unitary part $U$ of $V$ is trivial and $E$ is
injective, we conclude from the ``moreover'' parts of
Lemma~\ref{cuscik} and Corollary~\ref{strbli} that
$T_z$ is weakly stable. Hence, $T_z$ is a weakly
stable B-normal operator which is not block-diagonal
relative to $\hh_1\oplus \hh_2$. By choosing $E$ as a
non-zero multiple of an isometry $W\in
\ogr{\hh_2,\hh_1}$ with range orthogonal to $\ob{V}$,
we can guarantee that $\ob{|E|}$ is closed.
   \hfill $\diamondsuit$
   \end{exa}
Now we provide a concrete example of a
weakly stable B-normal operator $T = \big[\begin{smallmatrix} V & E \\
0 & X \end{smallmatrix}\big]$ for which $\ob{|E|}$ is
not closed. This is closely related to
Lemma~\ref{tywsr}.
   \begin{exa} \label{nitclyst}
If the hypothesis in Lemma~\ref{tywsr} that $\ob{|E|}$
is closed is deleted, the conclusion is no longer
valid. Indeed, let $\mu$ be a probability Borel
measure on $\bar{\mathbb{D}}$, the closed unit disk in
$\cbb$ centered at $0$, such that
   \begin{align} \label{pdif}
\mu(\tbb)=0 \quad \text{and} \quad \sup\{|z|\colon z
\in \supp\mu\}=1,
   \end{align}
where $\supp\mu$ stands for the closed support of
$\mu$. Set $\hh_1=\hh_2=L^2(\mu)$. Denote by $M_{\xi}$
the operator of multiplication on $L^2(\mu)$ by a
complex Borel function $\xi$ on $\bar{\dbb}$. Define
the continuous functions $\psi$ and $\varphi$ on
$\bar{\dbb}$ by
   \begin{align} \label{pdif2}
\psi(z)=(1-|z|)^{1/2} \quad \text{and}\quad
\varphi(z)=|z|^{1/2} \quad \text{for } z\in
\bar{\dbb}.
   \end{align}
Clearly, $M_{\psi}$ and $M_{\varphi}$ are bounded
commuting positive operators. It follows from
\eqref{pdif} that there exists a sequence
$\{z_n\}_{n=1}^{\infty}\subseteq \dbb \cap \supp\mu$
converging to a point in $\tbb \cap \supp\mu$. This
implies that $\dim L^2(\mu)=\aleph_0$. Thus, there
exist two isometries $V\in \ogr{\hh_1}$ and $W\in
\ogr{\hh_2,\hh_1}$ such that $\ob{V}\perp\ob{W}$.
Define the operator $E\in \ogr{\hh_2,\hh_1}$ by
$E=WM_{\psi}$ and set $X=M_{\varphi}$. Since
$|E|=M_{\psi}$, the operator $T = \big[\begin{smallmatrix} V & E \\
0 & X \end{smallmatrix}\big]$ is B-normal relative to
$\hh_1\oplus \hh_2$. Hence, $\{T^n\}_{n=1}^{\infty}
\subseteq \gnbh$ (see Example~\ref{hgaq}). Moreover,
by the first equality in \eqref{pdif} and
\eqref{pdif2}, $|E|$ is injective, so $E$ is injective
and $\hh_{22}=\hh_2$. We now claim that the sequence
$\{\|E_n\|\}_{n=1}^{\infty}$ is bounded. Indeed, it
follows from Lemma~\ref{xyzzyx}(iii) that for all
$f\in L^2(\mu)$ and $n\Ge 1$,
   \allowdisplaybreaks
   \begin{align*}
(E_n^*E_n f)(z) & = \Big(\sum_{j=0}^{n-1}X^{*j}X^j
E^*E f\Big)(z)
   \\
& = \sum_{j=0}^{n-1}|z|^j (1-|z|)f(z) = (1-|z|^n)f(z)
\quad \text{for $\mu$-a.e. } z\in \bar\dbb,
   \end{align*}
which implies that $\sup_{n\Ge 1} \|E_n\| \Le 1$. Thus
the condition (i) of Lemma~\ref{tywsr} is valid. In
turn, using the second equality in \eqref{pdif} we
deduce that $r(X)=\|X\|=1$. Hence by
Lemma~\ref{wikwi}(iv), the conditions (ii) and (iii)
of Lemma~\ref{tywsr} do not hold (thus $\ob{|E|}$ is
not closed). This together with Lemma~\ref{wikwi}(iii)
implies that there exists a unit vector $h\in \hh_2$
such that $\sum_{j=0}^{\infty} \is{X^{*j}X^j h}{h}=
\infty$. Using the moreover part of
Corollary~\ref{strbli}, we conclude that $T$ is weakly
stable if and only if the unitary part of $V$ is
weakly stable. To have a more concrete example of a
measure with the above properties, consider the
normalized planar Lebesgue measure on $\bar\dbb$.
Straightforward computations show that the condition
\eqref{pdif} holds and $\sum_{j=0}^{\infty}
\is{X^{*j}X^j h}{h}=\infty$, where $h(z)=1$ for every
$z\in \bar\dbb$.
   \hfill $\diamondsuit$
   \end{exa}
We conclude the paper with the following remark.
   \begin{rem}
It is worth pointing out that if $T=T_z$ is the
B-operator as in Example~\ref{hgaq} and $\mathscr{X}
\in \{\mathscr{N}, \mathscr{Q}, \mathscr{S},
\mathscr{H}$\}, then $\{T^n\}_{n=1}^{\infty} \subseteq
\gxbh$, each $T^n$ is not block-diagonal (use
Lemma~\ref{xyzzyx}), and $\text{\sc (wot)}
\lim_{n\to\infty} T^n=0$. The same conclusion can be
guaranteed for the B-operator $T= \big[\begin{smallmatrix} V & E \\
0 & X \end{smallmatrix}\big]$ appearing in
Example~\ref{nitclyst}.
   \hfill $\diamondsuit$
   \end{rem}

   \end{document}